\documentclass[11pt,a4paper]{amsart}
\usepackage{exscale,amssymb,amsmath,color,amsmath,amsthm,amsfonts,bbm,stmaryrd,mathrsfs}
\usepackage{graphicx,fullpage}
\usepackage{hyperref}
\usepackage{tikz}
\usetikzlibrary{shapes,shapes.misc,calc,intersections,patterns,decorations.markings,plotmarks}

\tikzstyle{vertex}=[circle,fill=black!75,minimum size=5pt,inner sep=0pt]
\tikzstyle{car}=[minimum size=8pt,inner sep=0pt]

\usepackage[linesnumbered,ruled,vlined]{algorithm2e}

\newcommand{\E}[1]{\ensuremath{\mathbb{E} \left[#1 \right]}}

\newcommand{\Prob}[1]{\ensuremath{\mathbb{P} \left(#1 \right)}}

\newcommand{\bP}{\ensuremath{\mathbb{P}}}

\newcommand{\var}[1]{\ensuremath{\mathrm{var} \left(#1 \right)}}

\newcommand{\NN}{\ensuremath{\mathbb{N}}}

\newcommand{\Po}{\ensuremath{\mathrm{Po}}}
\newcommand{\Geom}{\ensuremath{\mathrm{Geom}}}

\newcommand{\GGW}{\ensuremath{\mathrm{GGW}}}
\newcommand{\GGWinf}{\ensuremath{\mathrm{GGW}_{\infty}}}

\newcommand{\fl}[1]{\ensuremath{\lfloor #1 \rfloor}}

\newcommand{\equidist}{\ensuremath{\stackrel{d}{=}}}

\newenvironment{itemize*}%
  {\vspace{-0.3cm}%
  \begin{itemize}%
    \setlength{\itemsep}{0pt}%
    \setlength{\parskip}{0pt}}%
  {\end{itemize}}

\newenvironment{enumerate*}%
  {\vspace{-0.3cm}%
  \begin{enumerate}%
    \setlength{\itemsep}{0pt}%
    \setlength{\parskip}{0pt}}%
  {\end{enumerate}}

\newtheorem{thm}{Theorem}[section]
\newtheorem{lemma}[thm]{Lemma}

\newtheorem{prop}[thm]{Proposition}

\newtheorem{conj}[thm]{Conjecture}

\title{Parking on a random rooted plane tree}
\author{Qizhao Chen}
\address{The Queen's College, University of Oxford}
\email{qizhao.chen@queens.ox.ac.uk}

\author{Christina Goldschmidt}
\address{Department of Statistics and Lady Margaret Hall, University of Oxford}
\email{goldschm@stats.ox.ac.uk}

\date{\today}

\begin{document}

\begin{abstract}
In this paper, we investigate a parking process on a uniform random rooted plane tree with $n$ vertices. Every vertex of the tree has a parking space for a single car.  Cars arrive at independent uniformly random vertices of the tree.  If the parking space at a vertex is unoccupied when a car arrives there, it parks.  If not, the car drives towards the root and parks in the first empty space it encounters (if there is one).  We are interested in asymptotics of the probability of the event that all cars can park when $\fl{\alpha n}$ cars arrive, for $\alpha > 0$. We observe that there is a phase transition at $\alpha_c := \sqrt{2} -1$: if $\alpha < \alpha_c$ then the event has positive probability, whereas for $\alpha > \alpha_c$ it has probability 0.  Analogous results have been proved by Lackner and Panholzer~\cite{LacknerPanholzer}, Goldschmidt and Przykucki~\cite{Goldschmidt and Przykucki} and Jones~\cite{Jones} for different underlying random tree models.
\end{abstract}

\maketitle

\section{Introduction} \label{sec:intro}

In this paper, we continue the investigation of phase transitions for parking processes on random trees begun by Lackner and Panholzer~\cite{LacknerPanholzer}, Goldschmidt and Przykucki~\cite{Goldschmidt and Przykucki} and Jones~\cite{Jones}.

Consider a finite rooted tree, whose edges are directed towards the root.  Each vertex has a parking space for a single car.  A fixed number of cars arrive one by one, each at a uniformly random vertex in the tree, independently of the others. If the parking space at a vertex is unoccupied when a car arrives there, it parks.  If not, the car drives towards the root and parks in the first empty space it encounters (if there is one). If all vertices along the way already have been taken, the car has to leave the tree, and we say that the car cannot park. We will be interested in the probability that all cars can park.

This problem originates in the computer science literature, where Konheim and Weiss~\cite{KonheimWeiss-parkingPath} introduced it in order to model collisions in hash functions.  Their setting takes a directed path on $[n]=\{1,2,...,n\}$ with edges directed towards 1 and has each car picking an independent uniformly random preferred parking spot.  On non-degenerate trees, which is the setting of primary interest to us, a similar model has been studied in the context of the modelling of rainfall runoff from hillsides (see Jones~\cite{Jones} and the references therein): here, the vertices of the tree represent different positions on the hillside and the directed edges represent paths down which water can flow.  Each location has a certain capacity to absorb rainwater; once the threshold is exceeded, water flows futher down the hill; and the question is whether or not there will be runoff at the bottom.

Lackner and Panholzer~\cite{LacknerPanholzer} considered the parking problem on a uniform random rooted labelled (unordered) tree with $n$ vertices.  They proved that when a number $\fl{\alpha n}$ of cars arrive independently at uniformly random preferred parking spots, the probability that all cars can park undergoes a phase transition as we vary $\alpha$. A probabilistic explanation for this phase transition was given in Goldschmidt and Przykucki~\cite{Goldschmidt and Przykucki}, using the objective method. In this paper, we investigate the scenario of uniform random rooted plane (i.e.\ ordered) trees instead of unordered trees, with $\fl{\alpha n}$ cars again arriving at independent uniform vertices.  We show that there is an analogous phase transition in this setting.

To specify our tree model precisely, let $\NN = \{1,2,\ldots\}$ and let $\mathbb{U} = \cup_{n \ge 0} \NN^n$ be the \emph{Ulam--Harris tree}. (By convention $\NN^0 = \{\emptyset\}$.) Any element $\mathbf{u} \in \mathbb{U}$ is a word $\mathbf{u} = u_1u_2\ldots u_n$ for some $n \ge 0$ (where if $n=0$ the word is empty).  The \emph{parent} $p(\mathbf{u})$ of $\mathbf{u}$ is $u_1 u_2 \ldots u_{n-1}$ and the \emph{children} of $\mathbf{u}$ are $\mathbf{u}1, \mathbf{u}2, \ldots$. A \emph{rooted plane tree} $\mathrm{t}$ is a subset of $\mathbb{U}$ with the following properties:
\begin{itemize}
\item[(a)] $\emptyset \in \mathrm{t}$;
\item[(b)] if $\mathbf{u} \in \mathrm{t}$ then $p(\mathbf{u}) \in \mathrm{t}$;
\item[(c)] for every $\mathbf{u} \in \mathrm{t}$, there exists $0 \le k(\mathbf{u}) < \infty$ such that if $k(\mathbf{u}) = 0$ then $\mathbf{u}i \notin \mathrm{t}$ for all $i \in \NN$, and if $k(\mathbf{u}) > 0$ then $\mathbf{u}i \in \mathrm{t}$ if and only if $1 \le i \le k(\mathbf{u})$.
\end{itemize}
(The \emph{root} is $\emptyset$.) We take as a convention that the edges of $\mathrm{t}$ are directed from $\mathbf{u}$ to $p(\mathbf{u})$. Write $\mathcal{T}$ for the set of rooted plane trees.  The \emph{size} of $\mathrm{t} \in \mathcal{T}$ (i.e.\ its size as a subset of $\mathbb{U}$) is denoted by $|\mathrm{t}|$.  Write $\mathcal{T}_n$ for the subset of $\mathcal{T}$ consisting only of trees of size $n$. We note that $\mathcal{T}_n$ is a finite set, whose cardinality is given by the Catalan number $C_{n-1}= \frac{1}{n} \binom{2n-2}{n-1}$.  Throughout this paper, we write $T_n$ for a uniformly random element of $\mathcal{T}_n$.

Our main result is the following.  
\begin{thm} \label{thm:keyproblem}
Suppose we have $\lfloor \alpha n \rfloor$ cars wishing to park on $T_n$, and each car arrives at an independent uniform random vertex of the tree. Let $A_{n,\alpha}$ be the event that all the cars can park. Then 
\[
 \lim_{n \to \infty} \bP(A_{n,\alpha}) = 
 \begin{cases}
\frac{\sqrt{1-2\alpha - \alpha^2}}{(1-\alpha)^2\exp(\alpha)} & \text{ if $0 \le \alpha \le \sqrt{2}-1$,} \\
0 & \text{ if $\alpha > \sqrt{2}-1$.}
\end{cases}
\]
\end{thm}


If $\alpha > 1$ then we have more cars than parking spaces and so trivially $\Prob{A_{n,\alpha} } = 0$ for all $n$ sufficiently large.  So for the rest of the paper we will impose that $\alpha \le 1$.
Henceforth, we write \Geom($\frac{1}{2}$) for the geometric distribution with success probability $1/2$ (having probability mass function $(1/2)^{i+1}$ for $i \ge 0$ and, therefore, mean $1$), and $\mathrm{Po}(\lambda)$ for the Poisson distribution with parameter $\lambda > 0$.  We write $\GGW(\frac{1}{2})$ for the law of the family tree of a Galton--Watson branching process with \Geom($\frac{1}{2}$) offspring distribution.  Fix a tree $\mathrm{t} \in \mathcal{T}_n$.  Then the \GGW($\frac{1}{2}$) model generates $\mathrm{t}$ with probability $\prod_{\mathbf{u} \in \mathrm{t}} 2^{-k(\mathbf{u}) - 1}$.  Since $\sum_{\mathbf{u} \in \mathrm{t}} k(\mathbf{u}) = n-1$, this product is equal to $2^{-2n+1}$.  So, in other words, the $\GGW(\frac{1}{2})$ model generates each of the $C_{n-1}$ elements of $\mathcal{T}_n$ with the same probability, and so, in particular, $T_n$ has the same distribution as a \GGW($\frac{1}{2}$) tree conditioned to have $n$ vertices. We will make extensive use of this fact.  Let us also observe that the number of cars arriving at a single vertex of $T_n$ has $\mathrm{Binomial}(\fl{\alpha n}, \frac{1}{n})$ distribution, which converges in distribution to $\mathrm{Po}(\alpha)$ as $n \to \infty$.

We will prove Theorem~\ref{thm:keyproblem} in several stages.  The structure of the proof follows closely that of Theorem 1.1 in \cite{Goldschmidt and Przykucki}.  The basic idea is that the model of random tree plus random arrivals of cars possesses a local weak limit, on which one can more easily analyse the problem, and for which the probability that all cars can park is exactly the limit given in Theorem~\ref{thm:keyproblem}.  Conveniently, it turns out that the probability that all cars can park behaves continuously with respect to this notion of convergence, and so we are able to deduce Theorem~\ref{thm:keyproblem}.  

The limit model consists of a tree generated as follows: we take an infinite backbone (a copy of $\mathbb{N} = \{1,2,\ldots\}$, rooted at $1$) and at each point $k \ge 1$ we graft a pair of i.i.d.\ $\GGW(\frac{1}{2})$ trees (by identifying their roots with $k$).  We orient all edges towards the root. (We do not give this tree a proper Ulam--Harris labelling because that will be unnecessary in the sequel.) Now let an independent $\Po(\alpha)$ number of cars arrive at each vertex.  It is convenient to analyse the parking process on this infinite tree in different stages.  In Section~\ref{sec:ggw}, we examine the parking process on a single $\GGW(\frac{1}{2})$ tree with i.i.d.\ $\Po(\alpha)$ arrivals of cars.  In this setting, the random number of cars that reach the root satisfies a recursive distributional equation, which enables us to essentially completely analyse its generating function. In Section~\ref{sec:limitingModel}, we first argue that the claimed limiting model really is the local weak limit of the finite-$n$ model.  Then we use the conclusions of Section~\ref{sec:ggw} to derive results for the parking problem on the full limiting model, and  complete the proof of Theorem~\ref{thm:keyproblem} by confirming that the probability that all cars can park on our limiting model is indeed the limit of such probabilities for uniform random rooted plane trees as the number of vertices tends to infinity. 

Finally, in Section~\ref{sec:context}, we put our results into context and discuss future directions.

\section{Parking on a single geometric Galton--Watson tree}\label{sec:ggw}

We begin by presenting the key theorem of this section.
\begin{thm}
\label{thm:GeoGWcritical}
 Let $X$ denote the number of cars that visit the root of a $\GGW(\frac{1}{2})$ tree with, for some $\alpha \in (0,1)$, an independent $\Po(\alpha)$ number of cars initially arriving at each vertex.  Let $p := \Prob{X=0}$.
 \begin{enumerate}
  \item \label{thm:GeoGWcritical:smallAlpha}
  If $ 0 < \alpha \le \sqrt{2}-1$ then the probability generating function of $X$ is
   \[
    G(s) = \frac{1}{2} \left((1+\alpha)s+1-\alpha + \sqrt{((1+\alpha)s+1-\alpha)^2-4s\exp(\alpha(s-1))} \right) .
   \]
   Moreover, $p = 1 - \alpha$ and $\E{X} =  \frac{1}{2}(1+\alpha - \sqrt{1 - 2\alpha -\alpha^{2}})$.  
  \item
  \label{thm:GeoGWcritical:largeAlpha}
  If $\alpha > \sqrt{2}-1$ then we have $1-\alpha < p \le \frac{2}{(3+2\sqrt{2})\alpha+1}$.
 Moreover, the probability generating function of $X$ is
$$
G(s) = \begin{cases}
    \frac{1}{2} \left((2-p)s+p - \sqrt{((2-p)s+p)^2-4s\exp(\alpha(s-1))} \right) &\mbox{if $s \geq s_p$,}\\
    \frac{1}{2} \left((2-p)s+p + \sqrt{((2-p)s+p)^2-4s\exp(\alpha(s-1))} \right)  &\mbox{if $s < s_p$,}
    \end{cases}
$$
where $s_p = \frac{(2-p-\alpha p)-\sqrt{(2-p-\alpha p)^2-4\alpha p(2-p)}}{2\alpha(2-p)}$.
Furthermore, for $\alpha > \sqrt{2} - 1$ we have $\E{X} = \infty$.
 \end{enumerate}
\end{thm}

We observe that the quantity $\E{X}$ undergoes a discontinuous phase transition at $\alpha = \sqrt{2}-1$:
\begin{equation} \label{eqn:discphtr}
\E{X} =  \begin{cases}
\frac{1}{2} (1+\alpha - \sqrt{1 - 2\alpha - \alpha^{2}}) & \text{ for $\alpha \le  \sqrt{2}-1$} \\
\infty & \text{ for $\alpha > \sqrt{2}-1$} .
\end{cases}
\end{equation}
(Note that $\E{X} = 1/\sqrt{2}$ for $\alpha = \sqrt{2} - 1$.)

The rest of Section~\ref{sec:ggw} is devoted to the proof of Theorem~\ref{thm:GeoGWcritical}. In Section~\ref{subsec:general}, we prove some basic properties of $G$, while in Sections~\ref{subsec:small} and \ref{subsec:large} we deal with cases that $\alpha$ is below and above the threshold $\sqrt{2}-1$, respectively.

\subsection{Results for general $\alpha$}\label{subsec:general}
To begin with, we establish some elementary results that hold for all $\alpha\in(0,1)$.

For an integer $m$ we write $m^+ = \max\{m,0\}$.  By the recursive definition of a $\GGW(\frac{1}{2})$ tree, we obtain the following recursive distributional equation (RDE) for $X$:
\begin{equation} \label{eqn:rde}
X \equidist P + \sum_{i=1}^N (X_i -1)^+,
\end{equation}
where $P \sim \Po(\alpha)$, $N \sim \Geom(\frac{1}{2})$, $X_1, X_2, \ldots$ are i.i.d.\ copies of the (non-negative integer-valued) random variable $X$, and all of the random variables on the right-hand side are independent.

\begin{prop}
\label{prop:p0}
Let $\alpha \in (0,1)$. We have
\[
 p = \Prob{X=0} \geq \frac{1}{2}\exp(-\alpha) > 0.
\]
Moreover, either $\E{X}=\infty$ or $p = 1-\alpha$ (or both).
\end{prop}

\begin{proof}
 The inequality $p \geq \frac{1}{2}\exp(-\alpha)$ holds because if the $\GGW(\frac{1}{2})$ tree contains only the root and no cars choose to park there then clearly $X=0$. Hence, we must have 
 \[
  p \geq \Prob{N=0, P=0} = \frac{1}{2}\exp(-\alpha) > 0.
 \]
 For the second statement, we take expectations in \eqref{eqn:rde} and obtain
 \[
  \E{X} = \alpha + \E{X} - \Prob{X \ge 1},
 \]
 which implies the claim.
\end{proof}

Now let $G(s) = \E{s^X}$ for $s \ge 0$ be the probability generating function of $X$. We obtain from the RDE~\eqref{eqn:rde} that
\begin{align*}
G(s) = \E{s^P} \E{\E{s^{(X - 1)^+}}^N }  & = \exp(\alpha(s-1)) \cdot \frac{1}{2-\E{s^{(X-1)^+}} } \\
& = \frac{ \exp(\alpha(s-1))}{2-\E{s^{X-1}}+(s^{-1}-1)p}  \\
& = \frac{\exp(\alpha(s-1))}{2-s^{-1}G(s)+(s^{-1}-1)p}\ .
\end{align*}
Hence, $G(s)$ satisfies a quadratic equation:
\begin{equation} \label{eqn:quadratic}
G(s)^{2} - \left(\left(2-p\right)s+p\right)G(s) + s\exp(\alpha(s-1)) = 0 .
\end{equation}
Solving this equation, we straightforwardly obtain that there are two possible values for $G(s)$ for each $s \in (0,1)$:\ \begin{equation}
 \label{eq:Gplusminus}
  G(s) = \frac{1}{2} \left((2-p)s+p \pm \sqrt{((2-p)s+p)^2-4s\exp(\alpha(s-1))} \right).
 \end{equation}

For simplicity we write 
\begin{align*}
Q_{+}(s) & :=\frac{1}{2} \left((2-p)s+p + \sqrt{((2-p)s+p)^2-4s\exp(\alpha(s-1))} \right)\\
Q_{-}(s)& :=\frac{1}{2} \left((2-p)s+p - \sqrt{((2-p)s+p)^2-4s\exp(\alpha(s-1))} \right)
\end{align*}
for the two branches. Since we have an explicit construction of a solution to the RDE (\ref{eqn:rde}), $G$ must exist and so, in particular, for all $s \in (0,1]$,  we have
\begin{equation} \label{eq:discriminant}
((2-p)s+p)^2-4s\exp(\alpha(s-1)) \ge 0.
\end{equation}

\begin{lemma}
\label{lem:pNotTooSmall}
 For any $\alpha \in (0,1)$, we have $p \geq 1-\alpha$.
\end{lemma}

\begin{proof}
We differentiate \eqref{eqn:quadratic} to obtain
\[
2G(s)G'(s) - \left(\left(2-p\right)s+p\right)G'(s) -(2-p)G(s)+ (1+\alpha s)\exp(\alpha(s-1)) = 0, 
\]
i.e.
\begin{equation}
\label{eq:expectationIntermediate}
G'(s) = \frac{(2-p)G(s)- (1+\alpha s)\exp(\alpha(s-1))}{2G(s)-\left(\left(2-p\right)s+p\right)}.
\end{equation}
Since $X < \infty$ almost surely, $G(1) = 1$.  So as $s \to 1$, the limit of the denominator in \eqref{eq:expectationIntermediate} is $0$. Assume (for a contradiction) that $p < 1-\alpha$.  Then the limit of the numerator is $1 - p - \alpha$, which is strictly positive. Hence by Abel's Theorem, which states that $\E{X} = G'(1-)$, the expectation of $X$ is infinite in absolute value, and since $X$ is non-negative we must have $\E{X} = G'(1-)= \infty$.  But then necessarily $2G(s)-\left(\left(2-p\right)s+p\right) = 2(G(s) -s) + p(s-1) \downarrow 0$ as $s \uparrow 1$.  In particular, this implies that $G(s)-s$ converges to zero from above. But since $G(s) \le 1$ for $s \in [0,1]$, we must then have  $G'(1-) \le 1$, which contradicts $G'(1-)= \infty$. It follows that $p \geq 1-\alpha$.
\end{proof}

The next lemma specifies that the probability generating function $G$ must take the branch $Q_+$ at least in a neighbourhood of $s=0$.
\begin{lemma}
 \label{lem:Q+AtS=0}
 For all $\alpha \in (0,1)$ there exists some $\varepsilon_\alpha >0$ such that for $s \in (0,\varepsilon_\alpha)$ we have 
 \[
 G(s) = \frac{1}{2} \left((2-p)s+p + \sqrt{((2-p)s+p)^2-4s\exp(\alpha(s-1))} \right).
 \]
\end{lemma}
\begin{proof}
 By Proposition~\ref{prop:p0}, we have $G(0) = p \neq 0$.  But $\lim_{s \to 0} Q_-(s) = 0$ and so $G$ and $Q_-$ differ at 0. (One may also check that $\lim_{s \to 0} Q_+(s) = p$.)  Since $G$ is continuous, this completes the proof.
\end{proof}

\subsection{Small $\alpha$}\label{subsec:small}
The following lemma shows that, for $\alpha \leq \sqrt{2}-1$,  $G(s)$ takes the branch $Q_+$ throughout $[0,1]$.

\begin{lemma}
\label{lem:Q_+forSmallAlpha}
 If $\alpha \leq \sqrt{2}-1$ then
  \begin{equation}
 \label{eq:Q_+forSmallAlpha}
 G(s) = \frac{1}{2} \left((2-p)s+p + \sqrt{((2-p)s+p)^2-4s\exp(\alpha(s-1))} \right).
 \end{equation}
 for all $s \in [0,1]$.
\end{lemma}
\begin{proof}
 By the continuity of $G$ and Lemma \ref{lem:Q+AtS=0}, it suffices to show that for all $s \in (0,1)$, the branches $Q_+$ and $Q_-$ do not meet, i.e.\ that
 $$((2-p)s+p)^2-4s\exp(\alpha(s-1)) > 0.$$
Since $((2-p)s+p)^2-4s\exp(\alpha(s-1))$ is increasing in $p$ for $s \in (0,1)$, by Lemma \ref{lem:pNotTooSmall} it is enough to show that $g(s) := ((1+\alpha)s+(1-\alpha))^2-4s\exp(\alpha(s-1)) >0$ for all $s \in (0,1)$.
 
The first and second derivatives of $g$ are
 \begin{align*}
 g'(s) & = 2(1+\alpha)^2s+2(1-\alpha^2)-4(1+\alpha s)\exp(\alpha(s-1))\\
 \intertext{and}
 g''(s) & = 2(1+\alpha)^2-4(2\alpha+\alpha^2s)\exp(\alpha(s-1)).
 \end{align*}
Observe that $g(1)=g'(1)=0$ and $g''(1)=-2(\alpha+1+\sqrt{2})(\alpha+1-\sqrt{2})$. If $\alpha \leq \sqrt{2}-1$, then $g''(1) \geq 0$. 
 Since $g''(s)$ is a strictly decreasing function for $s \in [0,1]$, this then gives that $g''(s) >0$ for all $s\in(0,1)$. Hence $g'(s)$ is strictly increasing for $s\in(0,1)$.  But then since  $g'(1) = 0$, by continuity of $g'$ at 1 we must have $g'(s) <0$ for $s\in(0,1)$. Hence, $g(s)$ is strictly decreasing for $s\in(0,1)$, and it follows that $((1+\alpha)s+(1-\alpha))^2-4s\exp(\alpha(s-1)) >0$ for all $s \in (0,1)$, as required.
\end{proof}

The following lemma establishes the values of $p$ and $\E{X}$ for $\alpha \in (0,\sqrt{2}-1]$.
\begin{lemma}
 \label{lem:psmallAlpha}
 For all $\alpha \in (0,\sqrt{2}-1]$, we have $p = 1-\alpha$ and $\E{X} = \frac{1}{2}(1+\alpha - \sqrt{1 - 2\alpha -\alpha^{2}})$.
\end{lemma}

\begin{proof}
We differentiate the expression for $G(s)$ stated in Lemma~\ref{lem:Q_+forSmallAlpha} to find 
\begin{equation}\label{eq:Gdash}
   G'(s) = 1-\frac{p}{2}+\frac{1}{2}\cdot \frac{(2-p)^2s+p(2-p)-2(1+\alpha s)\exp(\alpha(s-1))}{\sqrt{((2-p)s+p)^2-4s\exp(\alpha(s-1))}}.
 \end{equation}
 As $s \uparrow 1$, the numerator $(2-p)^2s+p(2-p)-2(1+\alpha s)\exp(\alpha(s-1))$ tends to $2(1-p-\alpha)$ while the denominator tends to $0$ from above. (As shown in the proof of Lemma~\ref{lem:Q_+forSmallAlpha}, the denominator is strictly positive for $s\in(0,1)$.) If we had $p>1-\alpha$, $G'(s)$ would tend to $-\infty$ which gives $\E{X}=G'(1-)=-\infty$, a contradiction. It then follows from Lemma \ref{lem:pNotTooSmall} that $p=1-\alpha$.

In particular, \eqref{eq:Gdash} becomes
 \begin{equation}\label{eq:Gdashalpha}
   G'(s) = \frac{1}{2} \left(1+\alpha+ \frac{(1+\alpha)^2s+1-\alpha^2-2(1+\alpha s)\exp(\alpha(s-1))}{\sqrt{((1+\alpha)s+1-\alpha)^2-4s\exp(\alpha(s-1))}} \right). 
 \end{equation}
Taking limit as $s \uparrow 1$ in \eqref{eq:expectationIntermediate} and using L'H\^opital's rule, we obtain that
\begin{align*}
G'(1-)& = \lim_{s \uparrow 1} \frac{(2-p)G(s)- (1+\alpha s)\exp(\alpha(s-1))}{2G(s)-\left(\left(2-p\right)s+p\right)}\\
& = \lim_{s \uparrow 1} \frac{(1+\alpha)G'(s)- (2\alpha+\alpha^2s)\exp(\alpha(s-1))}{2G'(s)-(1+\alpha)}  =\frac{(1+\alpha)G'(1-)- (2\alpha+\alpha^2)}{2G'(1-)-(1+\alpha)}. 
\end{align*}
Hence, $G'(1-)$ solves
$$
2G'(1-)^2-2(1+\alpha)G'(1-)+\alpha(2+\alpha)=0.
$$
This has solutions $\frac{1}{2} (1+\alpha \pm \sqrt{1 - 2 \alpha -\alpha^{2}}).$  To determine the correct sign, we note that $X$ is stochastically increasing in $\alpha$.  Since $\frac{1}{2}(1+\alpha \pm \sqrt{1- 2\alpha - \alpha^{2}})$ has derivatives 
\[
\frac{1}{2}\left(1 \mp \frac{1+\alpha}{\sqrt{1 - 2\alpha -\alpha^{2}}}\right)
\]
respectively, the requirement that the derivative be non-negative yields that $\E{X} = \frac{1}{2} (1+\alpha - \sqrt{1 - 2\alpha -\alpha^{2}})$.
\end{proof}

\subsection{Large $\alpha$}\label{subsec:large}
This last subsection deals with the case $\alpha > \sqrt{2}-1$. 
\begin{lemma}\label{lem:pLargeAlpha}
For $\alpha > \sqrt{2}-1$, we have $p>1-\alpha$ and $\E{X} = \infty$.
\end{lemma}
\begin{proof}
Assume, for a contradiction, that $p=1-\alpha$.
Recall from the proof of Lemma~\ref{lem:Q_+forSmallAlpha} that for $g(s) = ((1+\alpha)s+(1-\alpha))^2-4s\exp(\alpha(s-1))$ we have $g(1)=g'(1)=0$ and
\[
g''(1)=-2(\alpha+1+\sqrt{2})(\alpha+1-\sqrt{2}),
\]
which is strictly negative when $\alpha > \sqrt{2}-1$. Hence, there exists some $\delta_\alpha >0$ such that for $s \in (1-\delta_\alpha,1)$, we have $g(s)<0$. This contradicts (\ref{eq:discriminant}) and the first result follows.  The second follows immediately from Proposition~\ref{prop:p0}.
\end{proof}

The next lemma shows that for $\alpha > \sqrt{2}-1$, $G(s)$ is equal to $Q_-$ in a neighbourhood of $1$ and so cannot be equal to $Q_+$ throughout the interval $[0,1]$.

\begin{lemma}\label{lem:BigAlpha-around1}
If $\alpha > \sqrt{2}-1$, then there exists $\delta_{\alpha}>0$ such that $$
G(s) = \frac{1}{2} \left( (2-p)s+p - \sqrt{((2-p)s+p)^2-4s\exp(\alpha(s-1))} \right) \quad \text{for $s \in (1-\delta_{\alpha},1)$.}
$$
\end{lemma}
\begin{proof}
Suppose, for a contradiction, that there exists $\delta > 0$ such that
$$
G(s) = \frac{1}{2} \left( (2-p)s+p + \sqrt{((2-p)s+p)^2-4s\exp(\alpha(s-1))} \right) \quad \text{for $s \in (1-\delta,1)$.}
$$
Then differentiating $G(s)$ would give 
$$G'(s) = 1-\frac{p}{2}+\frac{1}{2}\cdot \frac{(2-p)^2s+p(2-p)-2(1+\alpha s)\exp(\alpha(s-1))}{\sqrt{((2-p)s+p)^2-4s\exp(\alpha(s-1))}}.$$
The numerator of the fraction tends to the negative quantity $2(1-\alpha-p)$ as $s \uparrow 1$, while the denominator tends to $0+$ as $s \uparrow 1.$
But then $\E{X} = G'(1-) = - \infty$, which contradicts $\E{X} = \infty.$ 
\end{proof}

\begin{lemma}\label{lem:transition}
Let $\alpha > \sqrt{2}-1$. Then there exists $t\in(0,1)$ such that 
\[
G(s) =
\begin{cases}
\frac{1}{2} \left( (2-p)s+p + \sqrt{((2-p)s+p)^2-4s\exp(\alpha(s-1))} \right) \quad 
& \text{if $s<t$ }\\
\frac{1}{2} \left( (2-p)s+p - \sqrt{((2-p)s+p)^2-4s\exp(\alpha(s-1))} \right) \quad 
& \text{if $s\geq t$.}
\end{cases}
\]
\end{lemma}

\begin{proof}
First we analyse the roots of $h(s) := ((2-p)s+p)^2-4s\exp(\alpha(s-1))$ as a function of $s$ for $s\in(0,1)$. We note that for $G(s)$ to have a solution as at (\ref{eq:Gplusminus}), we must have $h(s) \ge 0$ for $s \in [0,1]$, and so any root of $h$ in $(0,1)$ must also be a turning point.  We obtain two equations for the root(s) to satisfy:
$$
\begin{cases}
    ((2-p)s+p)^2-4s\exp(\alpha(s-1)) = 0\\
  ((2-p)^2s+p(2-p))-2(1+\alpha s)\exp(\alpha(s-1)) = 0,
    \end{cases}
$$
where the second equation was obtained by differentiating the first.
Combining these, we obtain that
\begin{align*}
((2-p)s+p)^2(1+\alpha s) &= 4s(1+\alpha s)\exp(\alpha(s-1)) \\
& = 2s((2-p)^2s+p(2-p)).
\end{align*}
Dividing by $(2-p)s+p$ gives
$$((2-p)s+p)(1+\alpha s)= 2s(2-p),$$
and rearranging gives the quadratic equation
\begin{equation}\label{eq:rootntp}
   f(s) := \alpha (2-p)s^2 + (\alpha p+p-2)s+p=0,
\end{equation}
which has at most two roots in $(0,1)$.

Now by Lemmas~\ref{lem:Q+AtS=0} and \ref{lem:BigAlpha-around1} there must exist some $t \in (0,1)$ where the $Q_+$ and $Q_-$ meet and $G(s)$ changes branch. We note that $t$ must be a root of $f$. We have shown that there can be at most two such roots, which implies that $G(s)$ switches branch at most twice in $[0,1]$. But since $G(s)$ takes different branches near 0 and near 1, it must switch exactly once.
\end{proof}

With the results in the above proof, we pursue our analysis of the function $h(s) = ((2-p)s+p)^2-4s\exp(\alpha(s-1))$.

\begin{lemma}
Let $\alpha > \sqrt{2}-1$. Then $h$ has two turning points, $t_1$ and $t_2$, which are such that $0<t_1<t_2<1$. Both $t_1$ and $1$ are roots of $h$, and $t_1$ is the point where $G$ switches branch. Moreover, $h$ is monotonically decreasing on $(0,t_1)$, increasing on $(t_1,t_2)$, and decreasing on $(t_2,1)$.  
\end{lemma}
\begin{proof}
First we note that 
\[
h'(s)=2((2-p)^2s+p(2-p))-4(1+\alpha s)\exp(\alpha(s-1))
\]
and 
\[
h''(s)=2(2-p)^2-4(2\alpha+\alpha^2s)\exp(\alpha(s-1)).
\]
The second derivative $h''(s)$ is monotonically decreasing in $s$ on $(0,1)$, so that either $h'(s)$ is monotonic or has exactly one turning point. In either case, $h'(s)$ has at most two roots, which means that $h$ has at most two turning points in $(0,1)$. 

We have already established that $h$ stays positive on $(0,1)$ except for one or two points where $h$ and $h'$ are both equal to zero.  Let $t_1$ be the leftmost of these.  Then we have $h(0) = p^2 > 0$, $h(t_1) = 0$ and $h(1) = 0$.  In order to have $h(s) \ge 0$ on $(t_1,1)$, there must be a turning point $t_2 \in (t_1, 1)$ at which $h(t_2) > 0$.  Since $h$ cannot have three turning points in $(0,1)$, there can be no roots of $h$ in $(t_1, 1)$.  The description of $h$ given in the statement of the lemma follows.
\end{proof}

This also provides us with more information on the values of $p$ and the turning point $t$. 

\begin{lemma}
Let $\alpha > \sqrt{2}-1$. Then $1-\alpha < p \le \frac{2}{1+(3+2\sqrt{2})\alpha}$, and the point $t$ where $G(s)$ switches branch is the smaller root of the quadratic function $f$ from \eqref{eq:rootntp}.
\end{lemma}
\begin{proof}
It is clear that $t$ needs to be a root of $f$. Now $f(0)$ and $f(1)$ are both strictly positive, and there must also be a root in $(0,1)$ in order for $G$ to change branch.  So the discriminant $(\alpha p + p -2)^2 - 4\alpha p (2-p)$ must stay non-negative and both (including repeated) roots of $f$ need to lie in $(0,1)$. Now
\[
(\alpha p + p -2)^2 - 4\alpha p (2-p) = (6 - (3 + \alpha) p)^2 - 8 (2-p)^2,
\]
which is non-negative if and only if $(6 - (3+\alpha)p)^2 \ge 8 (2-p)^2$.  But since $6 - (3 + \alpha)p > 0$ and $2-p > 0$, we obtain
\[
6 - (3 + \alpha)p \ge 2 \sqrt{2} (2-p)
\]
and, rearranging gives the upper bound for $p$. The lower bound comes from Lemma \ref{lem:pLargeAlpha}. 

The conclusion that $t$ is the smaller root of $f$ comes from observing that, if there are two roots, then the other root, $r$, must satisfy $h(r)>0$, i.e.\ $4r\exp(\alpha(r-1)) <((2-p)r+p)^2$. Putting this into the expression for the derivative of $h$ we get 
\begin{align*}
h'(r) & >2((2-p)^2r+p(2-p))-\frac{1+\alpha r}{r}\cdot ((2-p)r+p)^2 \\
&\qquad  = -\frac{((2-p)r + p)}{r} \left[ \alpha (2-p)r^2 + (\alpha p+p-2)r+p \right]  = -\frac{((2-p)r + p)}{r} f(r).
\end{align*}
But $r$ is a root of $f$ and so the right-hand side is 0.  Hence, $h'(r) > 0$.  But by the previous lemma this is possible only if $r>t= t_1$.
\end{proof}

This completes the proof of Theorem~\ref{thm:GeoGWcritical}.

\section{The limiting model and the parking problem on it}\label{sec:limitingModel}

\subsection{The limiting model}\label{subsec: lmaal}

We wish to consider the \emph{local weak limit} of $T_n$ (following, for example, Aldous and Steele~\cite{AldousSteele}), which should intuitively be obtained by conditioning a \GGW($\frac{1}{2}$) tree on non-extinction.  It follows from Lemma 1.14 of Kesten~\cite{Kesten} that one can make sense of this (\emph{a priori} singular) conditioning for any mean-1 offspring distribution $\mu$.  In particular, the limit object may be constructed as follows.  Vertices may have one of two types: they are either special or non-special. The root is special. 
\begin{itemize}
\item[(1)] Special individuals reproduce according to the size-biased distribution $\mu^*(k) := k\mu(k)$, $k \geq 1$; non-special individuals reproduce according to the normal offspring distribution $\mu$.  All individuals reproduce independently.
\item[(2)] Among the offspring of a special individual, one child is chosen uniformly at random and declared to be special. The remaining children are non-special, and have independent numbers of children, each with distribution $\mu$.
\item[(3)] The offspring of a non-special individual are all non-special.
\end{itemize}

Since we do not care about the planar ordering of the different offspring, and the size-biased distribution associated with $\Geom(\frac{1}{2})$ is that of $1 + G_1 + G_2$ where $G_1$ and $G_2$ are independent  $\Geom(\frac{1}{2})$ random variables, we obtain the following simplification, which is our limiting tree model, $T$.
Start with an infinite directed path $\Pi_\infty$ on $\NN$, with edges directed from $k+1$ to $k$ for all $k \geq 1$. Then, for every $k$, graft two independent \GGW($\frac{1}{2}$) trees onto $k$ (by identifying their roots with $k$) with all edges directed towards $k$. Finally, root the resulting (infinite) tree at 1.  We call this model $\GGW_{\infty}(\frac{1}{2})$.

On the tree of size $n$, we had $\fl{\alpha n}$ cars arriving at i.i.d.\ uniform parking spots.  So, in particular, the joint distribution of the numbers of cars arriving at the different vertices is Multinomial($\lfloor \alpha n \rfloor$; $\frac{1}{n},\ldots,\frac{1}{n})$.  It is not hard to see that in the limit (at least as long as we restrict our attention to a fixed subset of the vertices) we get an independent $\Po(\alpha)$ number of cars at each vertex of the limiting tree model. 
 
In summary, we claim that the local weak limit of our model is given by a $\GGW_{\infty}(\frac{1}{2})$ tree $T$ with i.i.d.\ $\Po(\alpha)$ arrivals of cars at each vertex. Since, in fact, our only interest is in showing that the probability all cars can park on $T_n$ converges to the probability that all cars can park on $T$, we will not go into the topological details required to properly set up the convergence of the tree and cars.  We will prove the convergence of the probabilities in Section~\ref{sec:finishingProof} below.

\subsection{Parking on the limit model}\label{subsec:piinfty}

In Section \ref{sec:ggw}, we studied the random number $X$ of cars that arrive at the root of a single $\GGW(\frac{1}{2})$ tree. Based on this knowledge of $X$, we will re-examine the parking process on our limiting model. We find it convenient to do this in two steps: first let the cars which arrive within the geometric Galton--Watson subtrees hanging from $\Pi_{\infty}$ park within those subtrees if they can, or stop at the root of their subtree if they cannot.   Let $Y_i$ denote the numbers of cars wanting to park at $i \in \NN$. Then in the second step, consider the parking problem restricted to the infinite directed path $\Pi_{\infty}$ with arrivals of cars now given by $Y_i, i \in \NN$.

The random variables $Y_1, Y_2, \ldots$ are clearly i.i.d.  Since each vertex of $\Pi_{\infty}$ has a number of neighbours \emph{off} the path equal to the sum of two independent $\Geom(\frac{1}{2})$ random variables, we get that $Y_1, Y_2, \ldots$ have common distribution satisfying
\begin{equation} \label{eqn:yandx}
Y \equidist P + \sum_{i=1}^{N_1+N_2} (X_i -1)^+,
\end{equation}
where $P \sim \Po(\alpha)$, $N_1, N_2 \sim \Geom(\frac{1}{2})$, $X_1, X_2, \ldots$ are i.i.d.\ copies of $X$, and all of the random variables on the right-hand side are independent. 

As shown in Goldschmidt and Przykucki~\cite{Goldschmidt and Przykucki}, the cars can all park on $\Pi_\infty$ if and only if we have
\[
C_n = n - \sum_{k=1}^n Y_k \geq 0 \quad \text{ for all $n \in \NN$.}
\]
(One can think of $C_n$ as the ``spare parking capacity'' in the first $n$ vertices along the spine -- if this ever becomes negative then some car cannot park.)  The next result is the key theorem in this section.

\begin{thm}
\label{thm:infiniteParkingProb}
Let $T$ denote a $\GGWinf(\frac{1}{2})$ tree with all edges directed towards the root. Now initially place an independent $\Po(\alpha)$ number of cars at each vertex of the tree. Following our conventional parking rules, let $A_{\alpha}$ be the event that all cars can park on $T$. Then
\[
\bP( A_\alpha) = \begin{cases}
\frac{\sqrt{1-2\alpha-\alpha^2}}{(1-\alpha)^2\exp(\alpha)} & \text{ if $0 \le \alpha \le \sqrt{2}-1$,} \\
0 & \text{ if $\alpha > \sqrt{2}-1$.}
\end{cases}
\]
\end{thm}

In particular, we may reformulate
\[
\Prob{A_{\alpha}} = \Prob{C_n \geq 0 \text{ for all } n \geq 1}.
\]
\begin{proof}
The process $(C_n)_{n \ge 1}$ is a random walk with initial state $C_0 = 0$ and step-size $1-Y_n$.  If $\E{Y}>1$, then by the strong law of large numbers, 
\[
\Prob{C_n \geq 0 \text{ for all } n \geq 1} = 0.
\]
Now by \eqref{eqn:yandx},
\begin{equation} \label{eqn:EY}
\E{Y} = \alpha + 2\E{X}-2(1-p)  = \begin{cases}
    1-\sqrt{1-2\alpha-\alpha^2}\ &\mbox{if $\alpha \le \sqrt{2}-1$}\\
    \infty\  &\mbox{if $\alpha > \sqrt{2}-1$,}
    \end{cases}
\end{equation}
where the second line follows from Theorem~\ref{thm:GeoGWcritical}.  (Note that $\E{Y} = 1$ if $\alpha = \sqrt{2} - 1$.)  Hence, when $\alpha > \sqrt{2}-1$, clearly $\Prob{A_{\alpha}} = 0$. 

For $\alpha \le \sqrt{2} -1$, we apply the following lemma, whose proof may be found, for example, in \cite{BrownPekozRoss-skipFreeWalks}.

\begin{lemma} \label{lem:skipfree}
Let $(S_n)_{n \ge 1}$ be a random walk with i.i.d.\ step sizes $W_1, W_2, \ldots$ such that $\Prob{W_1\leq 1} = 1$, $\E{W_1}=m \geq 0$, $\Prob{W_1 = 1} = q>0$. Then
$$\Prob{S_n \geq 0 \text{ for all $n \in \NN$}} = \frac{m}{q}.$$
\end{lemma}

By \eqref{eqn:rde}, we have
$$
\Prob{X = 0} =\Prob{P = 0}\Prob{\sum_{i=1}^N (X_i -1)^+=0},
$$
where clearly $\Prob{P = 0} = \exp(-\alpha)$. Since $\alpha \leq \sqrt{2}-1$, we have $\Prob{X = 0} = p = 1-\alpha$, and so $$\Prob{\sum_{i=1}^N (X_i -1)^+=0}=(1-\alpha)\exp(\alpha).$$
But then $$\Prob{Y = 0} =\Prob{P = 0}\Prob{\sum_{i=1}^N (X_i -1)^+=0}^2 = (1-\alpha)^2\exp(\alpha)$$ and applying Lemma~\ref{lem:skipfree}, we obtain the claimed result.
\end{proof}

\subsection{Completing the proof of Theorem~\ref{thm:keyproblem}.}\label{sec:finishingProof}

Conditionally on the existence of a monotone coupling of the trees $T_n$ as $n$ varies, the derivation of Theorem~\ref{thm:keyproblem} from Theorem~\ref{thm:infiniteParkingProb} is now identical to the proof of Theorem 1.1 in Goldschmidt and Przykucki~\cite{Goldschmidt and Przykucki}, the analogous result for uniform random (unordered) tree.  The existence of such a coupling (also known as a ``building scheme'') follows from results of Luczak and Winkler~\cite{LuczakWinkler}; see, in particular, the discussion on p.427 of their paper.  We will give a short proof for the sake of completeness.

\begin{figure}
\includegraphics[width=4.5cm,page=1]{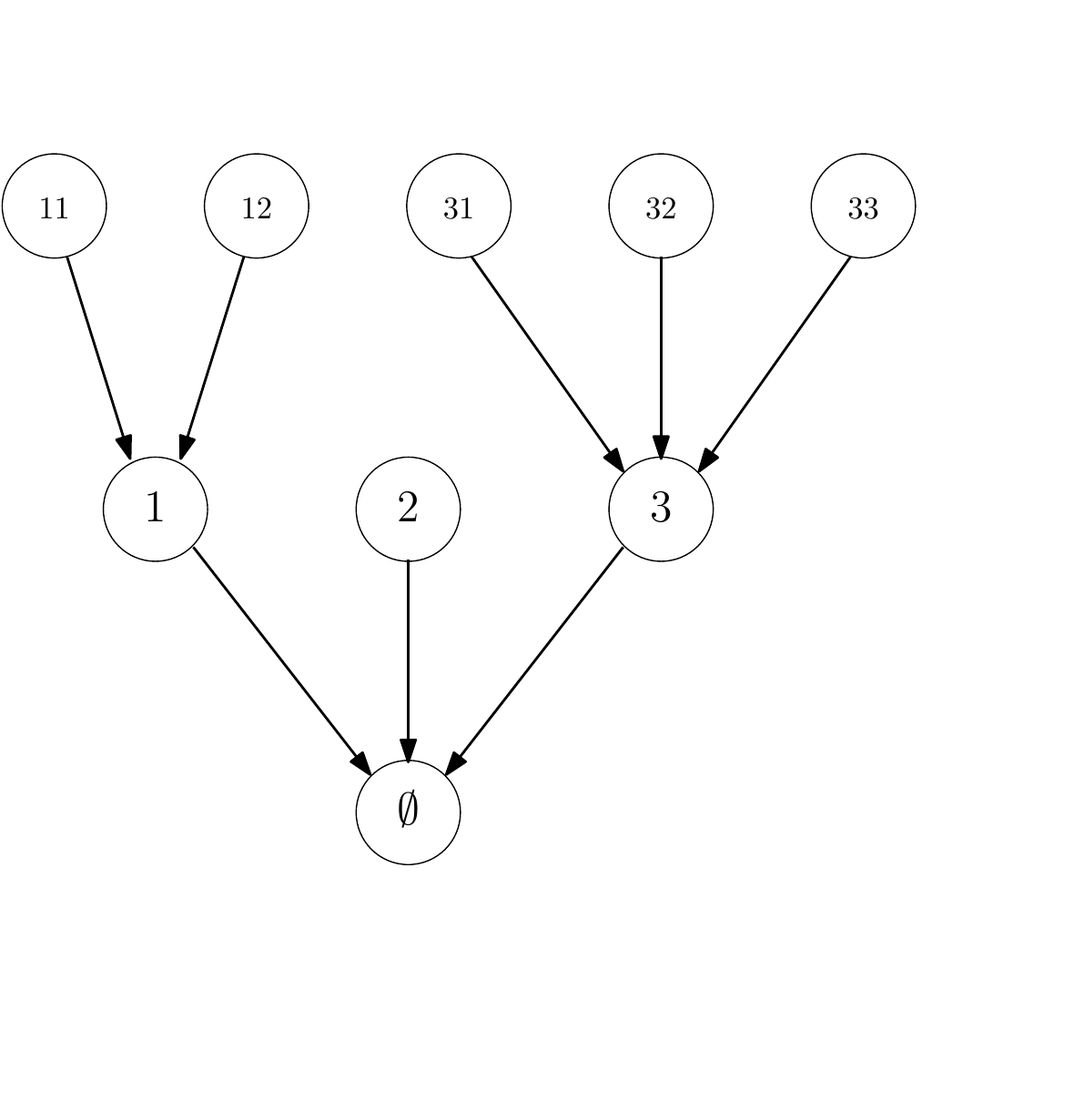} 
\includegraphics[width=4.5cm,page=2]{RotationCorrespondence.pdf} 
\includegraphics[width=4.5cm,page=3]{RotationCorrespondence.pdf} 
\caption{The rotation correspondence, $\Phi$, maps the plane tree on the left to the binary tree on the right.}
\label{fig:rotation}
\end{figure}

\begin{prop}[Luczak and Winkler~\cite{LuczakWinkler}] \label{prop:building}
There exists a coupling of the trees $(T_n)_{n \ge 1}$ which is such that $T_1$ consists of the root $\emptyset$ and, for each $n \ge 1$, $T_{n+1}$ is obtained from $T_n$ by the addition of a leaf.
\end{prop}

\begin{proof}
We start by introducing another model of trees.  We call a \emph{binary tree} a subset $\texttt{t} \in \cup_{n \ge 0} \{0,1\}^n$ (where, by convention, $\{0,1\}^0  = \emptyset$) satisfying the following conditions:
\begin{itemize}
\item $\emptyset \in \texttt{t}$;
\item if $\mathbf{u} \in \texttt{t}$ then $p(\mathbf{u}) \in \texttt{t}$.
\end{itemize}
Notice that each vertex $\mathbf{u}$ of the tree has at most two children: a \emph{left} child $\mathbf{u}0$, a \emph{right} child $\mathbf{u} 1$ or both. Let $\mathscr{T}^2$ be the set of binary trees, and let $\mathscr{T}_n^2$ be the subset of binary trees of size $n$.  Then $\mathscr{T}_n^{2}$ is enumerated by the Catalan number $C_n$.  There is a well-known bijection $\Phi: \mathcal{T}_n \to \mathscr{T}_{n-1}^{2}$, sometimes called the \emph{rotation correspondence}, which works as follows.  Let $\mathrm{t} \in \mathcal{T}_n$.  For each vertex $\mathbf{u} \in \mathrm{t}$, if $k(\mathbf{u}) \ge 1$ then keep the edge from $\mathbf{u}1$ to $\mathbf{u}$ and delete the edges from $\mathbf{u}2, \ldots, \mathbf{u}k(\mathbf{u})$ to $\mathbf{u}$.  Draw new edges $\mathbf{u}(i+1)$ to $\mathbf{u}i$ for $1 \le i \le k(\mathbf{u}) - 1$.  Now remove the root $\emptyset$ and the edge to its (now unique) child.  In the resulting object, treat edges of the original tree as edges from left children, and edges which we have added between siblings as from right children, and relabel accordingly, to give an element of $\mathscr{T}^2_{n-1}$.  See Figure~\ref{fig:rotation} for an example.

In particular, if $\texttt{T}_{n-1}$ is a uniformly random element of $\mathscr{T}^2_{n-1}$ then $\Phi^{-1}(\texttt{T}_{n-1})$ has the same distribution as $T_n$, a uniformly random element of $\mathcal{T}_n$.

Section 3 of \cite{LuczakWinkler} is devoted to showing that there exists a building scheme for $(\texttt{T}_n)_{n \ge 1}$ which works by always adding a leaf.  It is straightforward to see that addition of a leaf in $\texttt{t}$ corresponds to adding a new leaf (indeed, a leaf which is a rightmost child) in $\mathrm{t}$.  So if $(\texttt{T}_n)_{n \ge 1}$ is constructed according to the building scheme for binary trees then $(\Phi^{-1}(\texttt{T}_{n-1}))_{n \ge 2}$ is a building scheme for uniform random plane trees.
\end{proof}

This concludes the proof of Theorem~\ref{thm:keyproblem}.

\section{Our results in context}\label{sec:context}

Consider the following more general version of our model.  Take a Galton--Watson tree with offspring distribution the law of some random variable $N$, and assume that i.i.d.\ numbers of cars initially arrive at each of the vertices of the tree, with common distribution that of some random variable $P$.  Let us restrict attention to the situation where the offspring distribution is \emph{critical}, i.e.\ has mean 1, and \emph{non-degenerate} in the sense that $\Prob{N = 1} < 1$ .  (See \cite{Goldschmidt and Przykucki} for a discussion of the sub- and super-critical cases.)  As usual, let $X$ be the number of cars arriving at the root, which solves the analogue of the RDE (\ref{eqn:rde}). We are aware of three settings for which the distribution of $X$ has now been fully analysed: $N \sim \text{Po}(1)$, $P \sim \text{Po}(\alpha)$ in \cite{Goldschmidt and Przykucki}; $N \sim \Geom(\frac{1}{2})$, $P \sim \text{Po}(\alpha)$ in the present paper; and
$$
\Prob{N = 0} = \beta, \quad \Prob{N = 1} = 1-2\beta, \quad \Prob{N=2} = \beta,
$$
where $\beta \in (0,1/4]$ and 
$$
\Prob{P = 0} = 1 - \alpha/2, \quad \Prob{P = 2} = \alpha/2,
$$
where $\alpha \in (0,2]$, in \cite{Jones}.  In each case, we have a critical offspring distribution with finite variance and a stochastically monotone family of arrival distributions parameterised by $\alpha$; and in each case we observe a discontinuous phase transition for $\E{X}$, which jumps from a finite value to $\infty$ as $\alpha$ passes through some critical value $\alpha_c$.  In each case, we also observe the branch-switching phenomenon for the generating function of $X$ for $\alpha > \alpha_c$, as described in Section~\ref{subsec:large}.  In particular, the behaviour of the random variable $X$ seems to be, at least to some extent, universal. The following conjecture was made in \cite{Goldschmidt and Przykucki}.

\begin{conj}[Goldschmidt and Przykucki~\cite{Goldschmidt and Przykucki}]
Suppose that $\var{N} \le 1$, that $P$ is stochastically increasing in $\alpha = \E{P}$, and that $\var{P} < \infty$ for all $\alpha \ge 0$. Let $\nu(\alpha) = \E{P(P-1)}$ and define
\[
\alpha_c = \inf \left\{\alpha \ge 0: \alpha = 1 - \sqrt{\var{N} \nu(\alpha)}\right\}.
\]
Then
\[
\E{X} = 
\begin{cases}
\frac{1 - \alpha + \alpha \var{N} - \sqrt{(1-\alpha)^2 - \var{N} \nu(\alpha)}}{\var{N}} & \text{ if $\alpha \le \alpha_c$}, \\
\infty & \text{ if $\alpha > \alpha_c$}.
\end{cases}
\]
\end{conj}

This conjecture holds in our setting, despite the fact that our \Geom($\frac{1}{2}$) offspring distribution has $\var{N} = 2$.

\section{Acknowledgements}
Q.C.'s research was generously supported by Undergraduate Research Bursary URB-1718-42 from the London Mathematical Society, and by funding from the Mathematical Institute, University of Oxford.  C.G.'s research was supported by EPSRC Fellowship EP/N004833/1.

\bibliographystyle{abbrv}

\begin{thebibliography}{10}
\bibitem{AldousSteele}
D.~Aldous and J.~Steele. 
\newblock The objective method: Probabilistic combinatorial optimization and local weak convergence.
\newblock In {\em Probability on Discrete Structures} 
\newblock (H. Kesten, ed.), Vol. 110 of Encyclopaedia of Mathematical Sciences, Springer, 1--72, 2004.

\bibitem{BrownPekozRoss-skipFreeWalks}
M.~Brown, E.~Pek\"oz, and S.~Ross.
\newblock Some results for skip-free random walk.
\newblock {\em Probab. Eng. Inform. Sc.}, 24:491--507, 2010.

\bibitem{Goldschmidt and Przykucki}
C.~Goldschmidt and M.~Przykucki.
\newblock Parking on a random tree.
\newblock {\em Combin. Probab. Comput.}, 28(1):23--45, 2019.

\bibitem{Jones}
O.~Jones.
\newblock Runoff on rooted trees.
\newblock {\em J. Appl. Probab.}, to appear (\texttt{arXiv:1807.08803}), 2019.

\bibitem{Kesten}
H.~Kesten.
\newblock Subdiffusive behavior of random walk on a random cluster.
\newblock {\em Ann. Inst. H. Poincar{\'e} Probab. Statist.}, 22(4):425--487,
  1986.

\bibitem{KonheimWeiss-parkingPath}
A.~Konheim and B.~Weiss.
\newblock An occupancy discipline and applications.
\newblock {\em SIAM J. Appl. Math.}, 14:1266--1274, 1966.

\bibitem{LacknerPanholzer}
M.-L. Lackner and A.~Panholzer.
\newblock Parking functions for mappings.
\newblock {\em J. Combin. Theory Ser. A}, 142:1--28, 2016.

\bibitem{LuczakWinkler}
M.~Luczak and P.~Winkler.
\newblock Building uniformly random subtrees.
\newblock {\em Random Structures Algorithms}, 24(4):420--443, 2004.

\end{thebibliography}

\end{document}